\documentclass[12pt,reqno]{amsart}

\newcommand\version{June 19, 2017}

\usepackage{amsmath,amsfonts,amsthm,amssymb,amsxtra}

\newtheorem{theorem}{Theorem}%[section]
\newtheorem{proposition}[theorem]{Proposition}
\newtheorem{lemma}[theorem]{Lemma}
\newtheorem{corollary}[theorem]{Corollary}

\theoremstyle{definition}

\theoremstyle{remark}

\renewcommand{\epsilon}{\varepsilon}

\newcommand{\loc}{{\rm loc}}

\newcommand{\N}{\mathbb{N}}

\renewcommand{\phi}{\varphi}
\newcommand{\R}{\mathbb{R}}

\newcommand{\Sph}{\mathbb{S}}

\DeclareMathOperator{\diam}{diam}
\DeclareMathOperator{\dist}{dist}

\DeclareMathOperator{\supp}{supp}

\begin{document}

\title[A `liquid-solid' phase transition --- \version]{A `liquid-solid' phase transition in a simple\\ model for swarming, based on the `no flat-spots' theorem for subharmonic functions}

\author{Rupert L. Frank}
\address{Rupert L. Frank, Mathematisches Institut, Ludwig-Maximilans Universit\"at M\"un\-chen, Theresienstr. 39, 80333 M\"unchen, Germany, and Department of Mathematics, California Institute of Technology, Pasadena, CA 91125, USA}
\email{rlfrank@caltech.edu}

\author{Elliott H. Lieb}
\address{Elliott H. Lieb, Departments of Mathematics and Physics, Princeton
University, Princeton, NJ 08544, USA}
\email{lieb@princeton.edu}

\begin{abstract}
We consider a family of non-local shape optimization problems, which are motivated by a simple model for swarming and other self-assembly/aggregation models, and prove the existence of different phases for several of them. A technical key ingredient, which we establish, is that a strictly subharmonic function cannot be constant on a set of positive measure.
\end{abstract}

\maketitle

\renewcommand{\thefootnote}{${}$} \footnotetext{\copyright\, 2017 by the authors. This paper may be reproduced, in its entirety, for non-commercial purposes.\\
\emph{Date:} \version}

\section{Introduction and main result}

We are interested in the following minimization problem, depending on a parameter $\alpha>0$, which was recently introduced by Burchard, Choksi and Topaloglu in \cite{BCT}. For measurable functions $\rho\geq 0$ on
$\R^3$ one sets
$$
\mathcal E_\alpha[\rho] = \frac{1}{2}\iint_{\R^3\times\R^3} \rho(x) \left( \frac{1}{|x-y|} + |x-y|^\alpha \right) \rho(y)\,dx\,dy
$$
and, for $m>0$,
$$
E_\alpha(m) =\inf\left\{ \mathcal E_\alpha[\rho]:\ 0\leq \rho\leq 1\,,\ \int_{\R^3}\rho(x)\,dx = m \right\}.
$$

This is a simple model problem for flocking of birds or some
other  condensation phenomenon. The function $\rho$ describes the density of birds (or 'particles'). The energy functional $\mathcal
E_\alpha[\rho]$ has two terms. The first, $|x-y|^{-1}$, is a two-body
repulsive interaction between pairs of birds or particles. The second,
$|x-y|^\alpha $ is a two-body attractive interaction that engenders
condensation (or 'flocking').   The condition that $\rho \leq 1$,
introduced in \cite{BCT},  is a
many-body hard-core repulsion at short range. It imposes a maximum density,
beyond which the birds would be crushed. Its analogy in statistical
physics  is a bound on the allowed density of atoms in a liquid, namely the
density of the solid state. For background on mathematical models for biological aggregations we refer, for instance, to \cite{A,BT,FHK} and references therein.

We know from \cite{CFT} that for any $\alpha>0$ and $m>0$, the problem $E_\alpha(m)$ has a minimizer.

It is natural to think of the following three phases of the model, as described by the level set $\{\rho=1\}$ of a minimizer $\rho$ of the $E_\alpha(m)$ problem: 
\begin{align*}
& \text{Phase 1:} \qquad |\{\rho=1\}|=0  \,,\\
& \text{Phase 2:} \qquad 0<|\{\rho=1\}|<m \,, \\
& \text{Phase 3:} \qquad |\{\rho=1\}|=m \,.
\end{align*}
(Since we do not know whether minimizers are unique (modulo translations), it is possible to have mixtures of these phases.) We think of phase 1 as a `liquid phase', phase 3 as a `solid phase' and of phase 2 as an `intermediate phase'.

The following two theorems establish rigorously the existence of phases 1 and 3.

\begin{theorem}\label{main1}
For any $\alpha>0$ there is an $m_{c_1}(\alpha)>0$ such that for $m< m_{c_1}(\alpha)$ any minimizer $\rho$ for $E_\alpha(m)$ satisfies $|\{\rho=1\}|=0$.
\end{theorem}

\begin{theorem}\label{main}
For any $\alpha>0$ there is an $m_{c_2}(\alpha)<\infty$ such that for $m> m_{c_2}(\alpha)$ any minimizer $\rho$ for $E_\alpha(m)$ satisfies $|\{\rho=1\}|=m$.
\end{theorem}

Our results do not establish the existence of an intermediate phase 2. In the exactly solvable case $\alpha=2$ \cite{BCT} it is shown that phase 2 does not occur, but we believe that this is an un-typical behavior. (This belief is supported by the candidates from \cite{FH} for minimizers for $E_\alpha(m)$ when $m$ is small and $\alpha\neq 2$, which are \emph{not} characteristic functions. Note that the computations in \cite{FH} do not impose the constraint $\rho\leq 1$. Therefore, by the arguments in the proof of Lemma \ref{energysmall} below, they are relevant for our problem for small $m$.)

Results similar to ours were obtained in \cite{BCT} for kernels of the form $|x|^{-p} + |x|^2$ with $1<p<3$ in the analogue of Theorem \ref{main1} and $0<p<3$ in the analogue of Theorem \ref{main}. The proofs of these results, however, rely heavily on the algebraic properties of $|x|^2$. It is conjectured in \cite{BCT} that these qualitative facts should be true for a larger class of interaction kernels and our results confirm this expectation in another class of kernels.

Our results are not restricted to Coulomb singularities $|x|^{-1}$. For example, we can extend Theorem \ref{main} to the case of interaction kernels of the form
$$
|x|^{-\beta} + |x|^\alpha
\qquad\text{with}\ 0<\beta<1 \ \text{and}\ \alpha>0 \,.
$$
We explain this in Subsection \ref{sec:extensions}. Neither for this extension nor for our main results do we need that the kernels have an exact power law behavior, but we prefer to stick to this model case in order to make the arguments as simple as possible.

We end this introduction with the discussion of a related shape optimization problem, which was the main focus of the work of Burchard, Choksi and Topaloglu in \cite{BCT}. For $m>0$ one sets
$$
I_\alpha(m) =\inf\left\{ \mathcal E_\alpha[\chi_\Omega]:\ \Omega\subset\R^3\,,\ |\Omega| = m \right\}.
$$
As a consequence of Theorems \ref{main1} and \ref{main} we obtain the existence of a phase transition with respect to the parameter $m$.

\begin{corollary}\label{main0}
For any $\alpha>0$ there are $0<\tilde m_{c_1}(\alpha)\leq \tilde m_{c_2}(\alpha)<\infty$ such that $I_\alpha(m)$ has a minimizer for $m> \tilde m_{c_2}(\alpha)$ and has no minimizer for $m< \tilde m_{c_1}(\alpha)$.
\end{corollary}

It is natural to conjecture that $\tilde m_{c_1}(\alpha)= \tilde m_{c_2}(\alpha)$ and to wonder whether minimizers, whenever they exist, are spherically symmetric. These properties are true for $\alpha=2$, where the model is explicitly solvable \cite{BCT}. The following proof gives $m_{c_1}(\alpha)\leq \tilde m_{c_1}(\alpha)$ and $\tilde m_{c_2}(\alpha)\leq m_{c_2}(\alpha)$. (The strict inequality  $m_{c_1}(\alpha)<\tilde m_{c_1}(\alpha)$ occurs if for some $m\in ( m_{c_1}(\alpha), \tilde m_{c_1}(\alpha))$ all minimizers $\rho$ of the problem $E_\alpha(m)$ satisfy $|\{0<\rho<1\}|>0$. Similarly, the strict inequality $\tilde m_{c_2}(\alpha)< m_{c_2}(\alpha)$ would occur if for some $m\in (\tilde m_{c_2}(\alpha), m_{c_2}(\alpha))$ the problem $E_\alpha(m)$ had both a minimizer which is a characteristic function and one that is not.)

\begin{proof}[Proof of Corollary \ref{main0}]
We clearly have $E_\alpha(m)\leq I_\alpha(m)$. The important observation from \cite{BCT} is that, in fact,
$$
E_\alpha(m) = I_\alpha(m) \,.
$$
Moreover, the problem $I_\alpha(m)$ has a minimizer if and only if the characteristic function of a set is a minimizer for the relaxed problem $E_\alpha(m)$. Thus, according to Theorem \ref{main}, the $I_\alpha(m)$ problem has a minimizer for $m>m_{c_2}(\alpha)$ and, according to Theorem \ref{main1}, has no minimizer for $m<m_{c_1}(\alpha)$.
\end{proof}

\subsection*{Acknowledgements}

The authors are grateful to Luis Silvestre and to Mikhail Sodin for showing us how to prove Proposition \ref{shcont} for continuous and for $L^1$ functions, respectively, to Almut Burchard and Paul Gauthier for their careful reading and many helpful comments and to Haim Br\'ezis for references related to Proposition~\ref{sobconst}. Support by U.S. National Science Foundation grants DMS-1363432 (R.L.F.) and PHY-1265118 (E.H.L.) is acknowledged.

%%%%%%%%%%%%%%%%%%%%%%%%%%%%%%%%%%%%%%%%%%%%%%%%%%%%%%%%%

\section{Weak derivatives on sets of constancy}

The following result about functions in Sobolev spaces will play an important role in our proof.

\begin{proposition}\label{sobconst}
Let $\Omega\subset\R^d$ be an open set, $k\in\N$ and $u\in W^{k,1}_\loc(\Omega)$ real-valued. Then for any Borel set $A\subset\R$ of zero measure and any multi-index $\alpha\in\N_0^d$ with $0<|\alpha|\leq k$,
$$
\partial^\alpha u = 0
\qquad\text{almost everywhere on}\ u^{-1}(A) \,.
$$
\end{proposition}

\begin{proof}
For $k=1$ this is a classical result \cite{SV,MM,AL} (see also \cite[Thm. 6.19]{LL} for a textbook proof) and we now show that this implies the general result by a simple induction argument. Thus, let $k\geq 2$ and $0<|\alpha|\leq k$ and write $\partial^\alpha = \partial_j\partial^\beta$ for some $1\leq j\leq d$ and some multi-index $\beta$ with $0\leq|\beta|\leq k-1$. By induction, we have $v:=\partial^\beta u=0$ almost everywhere on $u^{-1}(A)$, so $v^{-1}(\{0\})\supset u^{-1}(A)$ (up to sets of measure zero). Moreover, since $v\in W^{1,1}_\loc(\Omega)$, we have again by the Almgren--Lieb result $\partial_j v=0$ almost everywhere on $v^{-1}(\{0\})$. In particular, $\partial_j v=0$ almost everywhere on $u^{-1}(A)$, which proves the assertion.
\end{proof}

From this proposition we deduce, in particular, that $\Delta u=0$ almost everywhere on $\{u=\tau\}$ which leads immediately to the following interesting corollary about strictly subharmonic functions.

\begin{corollary}[Strictly subharmonic functions have no flat spots]\label{sh}
Let $\Omega\subset\R^d$, $d\geq 2$, be an open set and assume that $u\in W^{2,1}_\loc(\Omega)$ satisfies $-\Delta u\leq -\epsilon$ in $\Omega$ for some $\epsilon>0$. Then $|\{ x\in\Omega: u(x) =\tau\}|=0$ for any $\tau\in\R$. 
\end{corollary}

The need of some strict subharmonicity assumption to deduce the absence of flat spots can be seen from the example $u(x) = (x_1)_+$, which is subharmonic and constant on the half-space $\{x_1\leq 0\}$.

The conclusion of the corollary remains valid if the assumption $u\in W^{2,1}_\loc(\Omega)$ is replaced by continuity and the equation $-\Delta u\leq -\epsilon$ is understood in viscosity sense. In fact, it remains valid for $u\in L^1_{\loc}(\Omega)$ under the assumption that Lebesgue measure is absolutely continuous with respect to the measure $\Delta u$. We will not need these results but, since they might be of independent interest, we present their proofs in two appendices.

%%%%%%%%%%%%%%%%%%%%%%%%%%%%%%%%%%%%%%%%%%%%%%%%%%%%%%%%%

\section{The Euler--Lagrange equation and the `chemical potential'}

In order to emphasize the general nature of the arguments in this section, we consider more general interaction kernels $k$ which are
\begin{equation}
\label{eq:kernelgen}
\begin{split}
& \text{locally integrable, non-negative, lower semi-continuous,}\\
& \text{and satisfy}\ \lim_{|x|\to\infty} k(x)=\infty \,.
\end{split}
\end{equation}
We set
\begin{equation}
\label{eq:energyfcnlgen}
\mathcal E[\rho] = \frac{1}{2}\iint_{\R^3\times\R^3} \rho(x) k(x-y) \rho(y)\,dx\,dy
\end{equation}
and
\begin{equation}
\label{eq:energygen}
E(m) =\inf\left\{ \mathcal E[\rho]:\ 0\leq \rho\leq 1\,,\ \int_{\R^3}\rho(x)\,dx = m \right\}.
\end{equation}
Existence of minimizers has been proved under the above conditions on $k$ in \cite{CCP,SST} for the problem without the $L^\infty$ constraint. The case of the $L^\infty$ constraint is, in fact, simpler. Moreover, the assumption of spherical symmetry of $k$ in \cite{SST} is not necessary.

Let $\rho$ be a minimizer for $E(m)$ and let
\begin{equation}
\label{eq:potgen}
\phi(x) = \int_{\R^3} k(x-y) \rho(y)\,dy
\end{equation}
be its potential. In \cite{BCT} it is shown that there is a constant $\mu>0$ such that for almost every $x\in\R^3$
\begin{align}\label{eq:eeq}
\phi(x) \leq\mu & \qquad \text{if}\ \rho(x)=1 \,,\notag\\
\phi(x) = \mu & \qquad\text{if}\ 0<\rho(x)<1 \,,\\
\phi(x) \geq \mu & \qquad\text{if}\ \rho(x)=0 \,. \notag
\end{align}

We now identify $\mu$ with the `chemical potential' (that is, the derivative of $E$ with respect to $m$) of the minimization problem.

\begin{lemma}\label{chempot}
Let $\rho$ be a minimizer of $E(m)$ for some $m>0$ and let $\mu$ be as in \eqref{eq:eeq}. Then
$$
\limsup_{m'\downarrow m} \frac{E(m')-E(m)}{m'-m} \leq \mu \leq \liminf_{m'\uparrow m} \frac{E(m')-E(m)}{m'-m} \,.
$$
\end{lemma}

Note that this implies, in particular, that $E$ is a continuous function of $m$ and that the singular part of its distributional derivative is non-positive.

\begin{proof}
Let $\epsilon>0$. Since the potential $\phi$ is continuous (see, for instance, the remark before \cite[Thm. 4.4]{BCT}) the set $F_\epsilon=\{ \mu < \phi \leq \mu+\epsilon\}$ has positive measure. Since $\lim_{|x|\to\infty} \phi(x)=\infty$ (see, for instance, the proof of \cite[Lem. 4.3]{BCT}), the set $F_\epsilon$ is bounded. Note that according to \eqref{eq:eeq}, $\rho=0$ on $F_\epsilon$. For any $m<m'\leq m + |F_\epsilon|$ we consider $\tilde\rho:=\rho + ((m'-m)/|F_\epsilon|) \chi_{F_\epsilon}$ which clearly satisfies the constraints of the $E(m')$ problem. Moreover,
\begin{align*}
E(m') & \leq \mathcal E[\tilde\rho] \\
& = \mathcal E[\rho] + \frac{m'-m}{|F_\epsilon|} \iint_{\R^3\times\R^3} \chi_{F_\epsilon}(x)k(x-y) \rho(y)\,dx\,dy + \left( \frac{m'-m}{|F_\epsilon|} \right)^2 \mathcal E[\chi_{F_\epsilon}] \\
& = E(m) + \frac{m'-m}{|F_\epsilon|} \int_{F_\epsilon} \phi(x)\,dx + \left( \frac{m'-m}{|F_\epsilon|} \right)^2 \mathcal E[\chi_{F_\epsilon}] \\
& \leq E(m) + (\mu+\epsilon) (m'-m) + \left( \frac{m'-m}{|F_\epsilon|} \right)^2 \mathcal E[\chi_{F_\epsilon}] \,.
\end{align*}
Letting $m'\downarrow m$ we find
$$
\limsup_{m'\downarrow m} \frac{E(m')-E(m)}{m'-m} \leq \mu+\epsilon \,,
$$
and, since $\epsilon>0$ is arbitrary, we obtain the left inequality in the lemma.

To prove the reverse inequality we distinguish two cases. First, assume that $\inf\phi<\mu$. Then, for fixed $\epsilon>0$, we choose $F_\epsilon:=\{\mu-\epsilon\leq\phi<\mu\}$, which has positive measure and is bounded. Moreover, by \eqref{eq:eeq}, $\rho=1$ on $F_\epsilon$. For any $m>m'>m-|F_\epsilon|$ we consider $\tilde\rho = \rho - ((m-m')/|F_\epsilon|)\chi_{F_\epsilon}$ and bound, similarly as before
\begin{align*}
E(m') & \leq \mathcal E[\tilde\rho] \\
& = E(m) - \frac{m-m'}{|F_\epsilon|} \int_{F_\epsilon} \phi(x)\,dx + \left( \frac{m-m'}{|F_\epsilon|} \right)^2 \mathcal E[\chi_{F_\epsilon}] \\
& \leq E(m) - (\mu-\epsilon) (m-m') + \left( \frac{m-m'}{|F_\epsilon|} \right)^2 \mathcal E[\chi_{F_\epsilon}] \,.
\end{align*}
This implies (note $m'-m<0$)
$$
\liminf_{m'\uparrow m} \frac{E(m')-E(m)}{m'-m} \geq \mu-\epsilon \,,
$$
which proves the right inequality in the lemma.

Now assume that $\inf\phi=\mu$. In this case we simply choose $\tilde\rho =(m'/m)\rho$ for any $m'<m$ and obtain
\begin{align*}
E(m') & \leq \mathcal E[\tilde\rho] \\
& \leq \left(\frac{m'}{m}\right)^2 E(m) \\
& = E(m) - \frac{m^2-(m')^2}{2m^2} \int_{\R^d} \rho\phi\,dx  \\
& \leq E(m) - \frac{m^2-(m')^2}{2m} \ \mu \,.
\end{align*}
That is,
$$
\frac{E(m')-E(m)}{m'-m} \geq \frac{m+m'}{2m} \ \mu \,,
$$
which again implies the right inequality in the lemma.
\end{proof}

%%%%%%%%%%%%%%%%%%%%%%%%%%%%%%%%%%%%%%%%%%%%%%%%%%%%%%%%

\section{Diameter bound}\label{sec:diameter}

It is known \cite{BCT} that, even for the general interaction kernels of the previous section, the support of minimizers is bounded. An important ingredient in the proof of Theorems \ref{main1} and \ref{main} is a quantitative version of this result which controls the size of the support in terms of $m$. More precisely, we show that the diameter of the support of $\rho$ grows at most like $m^{1/3}$ for large $m$. We emphasize that, while the results in this subsection can be extended to more general interaction kernels, for the sake of simplicity we restrict ourselves to the case $k(x)=|x|^{-1} + |x|^\alpha$.

\begin{theorem}\label{diameter}
For any $\alpha>0$ there is a constant $C$ such that for any $m>0$ and any minimizer $\rho$ of $E_\alpha(m)$ one has
$$
\diam\supp\rho \leq C \max\{1, m^{1/3}\} \,.
$$
\end{theorem}

The proof will rely on two auxiliary lemmas which we state and prove next. With any $\rho$ we associate its potential
\begin{equation}
\label{eq:pot}
\phi(x) = \int_{\R^3} \left( \frac{1}{|x-y|} + |x-y|^\alpha\right) \rho(y)\,dy \,.
\end{equation}
The first lemma does not require $\rho$ to be a minimizer. In fact, this lemma can be used to prove the existence of a minimizer.

\begin{lemma}\label{ball}
For any $\rho\geq 0$ with $\int_{\R^3}\rho(y)\,dy=m$ we have
$$
\sup_{a\in\R^3} \int_{B_R(a)} \rho(y)\,dy \geq m - \frac{2\,\mathcal E_\alpha[\rho]}{m R^\alpha}
$$
\end{lemma}

\begin{proof}
We bound
\begin{align*}
\phi(x) & \geq \int_{\R^3\setminus B_R(x)} |x-y|^\alpha \rho(y)\,dy \\
& \geq R^\alpha \left( m - \int_{B_R(x)} \rho(y)\,dy \right) \\
& \geq R^\alpha \left(m- \sup_{a\in\R^3} \int_{B_R(a)} \rho(y)\,dy \right)
\end{align*}
and obtain
$$
\mathcal E_\alpha[\rho] = \frac12 \int_{\R^3} \rho(x)\phi(x)\,dx \geq \frac{m}{2} R^\alpha \left(m- \sup_{a\in\R^3} \int_{B_R(a)} \rho(y)\,dy \right),
$$
which is the claimed inequality.
\end{proof}

The second lemma provides an upper bound on the potential on the support of $\rho$. The method of proof is reminiscent of some arguments in geometric measure theory, see, for instance, \cite[Lem.~4]{LO}.

\begin{lemma}\label{upperbound}
Let $\rho$ be a minimizer for $E_\alpha(m)$ for some $m>0$ and let $x\in\R^3$ be a Lebesgue point of $\rho$ with $\rho(x)>0$. Then
$$
\phi(x) \leq \frac{6+\alpha}{3} \ \frac{E_\alpha(m)}{m} \,.
$$
\end{lemma}

\begin{proof}
For fixed $r>0$ we define $\chi_<:= \chi_{B_r(x)}$ and $\chi_>=1-\chi_<$. We consider
$$
\tilde\rho(y):= \rho(y/\ell_r) \chi_>(y/\ell_r)
\qquad\text{with}\qquad
\ell_r := \left( \frac{m}{\int \chi_> \rho\,dy} \right)^{1/3} \,,
$$
which satisfies $\int_{\R^3} \tilde\rho(y)\,dy = m$ and therefore, by optimality of $\rho$,
\begin{equation}
\label{eq:optimality}
\mathcal E_\alpha[\rho] \leq \mathcal E_\alpha[\tilde\rho] \,.
\end{equation}
With the notation
$$
D^{(\beta)}[\sigma]:= \frac{1}{2} \iint_{\R^3\times\R^3} \sigma(y)|y-y'|^\beta \sigma(y')\,dy\,dy' \,,
\qquad
\phi^{(\beta)}(y):= \int_{\R^3} |y-y'|^\beta \rho(y')\,dy' \,,
$$
we have
\begin{align*}
\mathcal E_\alpha[\tilde\rho] & = \ell_r^5 D^{(-1)}[\chi_>\rho] + \ell_r^{6+\alpha} D^{(\alpha)}[\chi_>\rho]\\
& = \ell_r^5 \left( D^{(-1)}[\rho] - \int_{B_r(x)} \rho(y)\phi^{(-1)}(y)\,dy + D^{(-1)}[\chi_<\rho] \right)\\
& \qquad + \ell_r^{6+\alpha} \left( D^{(\alpha)}[\rho] - \int_{B_r(x)} \rho(y)\phi^{(\alpha)}(y)\,dy + D^{(\alpha)}[\chi_<\rho] \right).
\end{align*}
Therefore \eqref{eq:optimality} becomes
\begin{align*}
& \ell_r^5 \int_{B_r(x)} \rho(y)\phi^{(-1)}(y)\,dy + \ell_r^{6+\alpha} \int_{B_r(x)} \rho(y)\phi^{(\alpha)}(y)\,dy \\
& \qquad \leq \left( \ell_r^5 - 1\right) D^{(-1)}[\rho] + \left(\ell_r^{6+\alpha}-1\right) D^{(\alpha)}[\rho] 
+ \ell_r^5 D^{(-1)}[\chi_<\rho] + \ell_r^{6+\alpha} D^{(\alpha)}[\chi_<\rho] \,.
\end{align*}
Since $\ell_r\geq 1$ we can bound the left side from below by
$$
\ell_r^5 \int_{B_r(x)} \rho(y)\phi^{(-1)}(y)\,dy + \ell_r^{6+\alpha} \int_{B_r(x)} \rho(y)\phi^{(\alpha)}(y)\,dy \geq 
\int_{B_r(x)} \rho(y)\phi(y)\,dy \,.
$$
On the other hand, since $\rho\leq 1$ we can bound the last two terms on the right side by
$$
D^{(-1)}[\chi_<\rho] \leq D^{(-1)}[\chi_<] = C_1 r^5 \,,
\qquad
D^{(\alpha)}[\chi_<\rho] \leq D^{(\alpha)}[\chi_<] = C_2 r^{6+\alpha} \,,
$$
for some constants $C_1$ and $C_2$. Therefore we obtain
$$
\frac{1}{|B_r(x)|}\int_{B_r(x)} \rho(y)\phi(y)\,dy \leq \frac{\ell_r^5 - 1}{|B_r(x)|} D^{(-1)}[\rho] + \frac{\ell_r^{6+\alpha}-1}{|B_r(x)|} D^{(\alpha)}[\rho] + C_1' \ell_r^5 r^2 + C_2' \ell_r^{6+\alpha} r^{3+\alpha} \,.
$$
We now want to let $r\to 0$. Since $x$ is a Lebesgue point of $\rho$ and since $\phi$ is continuous, we have
$$
\lim_{r\to 0} \frac{1}{|B_r(x)|}\int_{B_r(x)} \rho(y)\phi(y)\,dy = \rho(x)\phi(x) \,.
$$
On the other hand, we have
$$
\frac{\ell_r^3 -1}{|B_r(x)|} = \frac{1}{m - \int_{B_r(x)} \rho(y)\,dy} \ \frac{1}{|B_r(x)|} \int_{B_r(x)} \rho(y)\,dy \to \frac{\rho(x)}{m} \,.
$$
Since
$$
\frac{\ell_r^5 -1}{\ell_r^3 -1} \to \frac{5}{3}
\qquad\text{and}\qquad
\frac{\ell_r^{6+\alpha} -1}{\ell_r^3 -1} \to \frac{6+\alpha}{3} \,,
$$
we finally conclude that
$$
\rho(x)\phi(x) \leq \left( \frac 53 D^{(-1)}[\rho] + \frac{6+\alpha}{3} D^{(\alpha)}[\rho] \right) \frac{\rho(x)}{m} \,.
$$
Bounding $5\leq 6+\alpha$ and recalling that $\rho(x)\neq 0$ we obtain the lemma.
\end{proof}

\begin{proof}[Proof of Theorem \ref{diameter}]
If we choose
$$
R = \left( \frac{4 E_\alpha(m)}{m^2} \right)^{1/\alpha}
$$
and $\rho$ is a minimizer, then $m- 2\mathcal E_\alpha[\rho]/(mR^\alpha)\geq m/2$ and therefore Lemma \ref{ball} (and the continuity of $a\mapsto \int_{B_R(a)}\rho(y)\,dy$) implies that there is an $a\in\R^3$ such that
$$
\int_{B_R(a)} \rho(y)\,dy \geq m/2 \,.
$$
From this we conclude that for any $x\in\R^3$ with $|x-a|> (\sigma+1)R$, where $\sigma>0$ is a parameter to be determined later,
\begin{align*}
\phi(x) & \geq \int_{B_R(a)} |x-y|^\alpha \rho(y)\,dy \geq (|x-a|-R)^\alpha \int_{B_R(a)}\rho(y)\,dy \geq (|x-a|-R)^\alpha m/2 \\
& > \sigma^\alpha R^\alpha m/2 = \sigma^\alpha 2 E_\alpha(m)/m \,.
\end{align*}
We choose $\sigma= (1+\alpha/6)^{1/\alpha}$, so that $2 \sigma^\alpha = (6+\alpha)/3$, and then Lemma \ref{upperbound} implies that $\rho(x)=0$ whenever $|x-a|>(\sigma+1)R$. Thus,
$$
\diam\supp\rho \leq 2(\sigma+1) R = 2((1+\alpha/6)^{1/\alpha}+1) (4E_\alpha(m)/m^2)^{1/\alpha} \,.
$$
Finally, by computing with trial functions we obtain
$$
E_\alpha(m) \leq \mathcal E_\alpha[m \chi_{B_{(3/(4\pi))^{1/3}}}] = C_1 m^2
\qquad\text{if}\ m\leq 1 \,,
$$
(the radius has not been optimized) and
$$
E_\alpha(m)\leq \mathcal E_\alpha[\chi_{B_{(3m/(4\pi))^{1/3}}}] = C_2 m^{5/3} + C_3 m^{2+\alpha/3} 
\qquad\text{if}\ m>0 \,,
$$
which implies that $E_\alpha(m)/m^2 \leq C_4 \max\{1, m^{\alpha/3}\}$. Inserting this into the diameter bound, we obtain the theorem.
\end{proof}

The following consequence of Theorem \ref{diameter} will be used in the proof of Theorems \ref{main1} and \ref{main}.

\begin{corollary}\label{diamcor}
Let $\rho$ be a minimizer for $E_\alpha(m)$ and $x\in\supp\rho$. Then, if $\alpha\leq 2$,
$$
\min\left\{m, m^{(\alpha+1)/3}\right\}\lesssim \int_{\R^3} |x-y|^{\alpha-2}\rho(y)\,dy \lesssim m^{(\alpha+1)/3}
$$
and, if $\alpha\geq 2$,
$$
m^{(\alpha+1)/3}\lesssim \int_{\R^3} |x-y|^{\alpha-2}\rho(y)\,dy \lesssim \max\left\{m,m^{(\alpha+1)/3}\right\}
$$
with implicit constants depending only on $\alpha$.
\end{corollary}

\begin{proof}
By the `bathtub principle' \cite[Theorem 1.14]{LL} we have, since $0\leq\rho\leq 1$ and $\int \rho\,dy = m$,
$$
\int_{\R^3} |x-y|^{\alpha-2}\rho(y)\,dy \leq \int_{B_{(3m/(4\pi))^{1/3}}(x)} |x-y|^{\alpha-2}\,dy = C m^{(\alpha+1)/3} \qquad\text{if}\ \alpha\leq 2
$$
and
$$
\int_{\R^3} |x-y|^{\alpha-2} \rho(y)\,dy \geq \int_{B_{(3m/(4\pi))^{1/3}}(x)} |x-y|^{\alpha-2}\,dy = C m^{(\alpha+1)/3} \qquad\text{if}\ \alpha\geq 2 \,.
$$
(Note that these inequalities are valid for \emph{any} $x\in\R^3$.)

We now prove the opposite inequalities. Let $d=C \max\{1,m^{1/3}\}$, where $C$ is the constant from Theorem \ref{diameter}, so that the support of $\rho$ is contained in a ball $B$ of radius $d$. We shall prove the inequalities in the corollary for every $x\in B$ and therefore, in particular, for all $x\in\supp\rho$. Note that for any $x\in B$, $\supp\rho$ is contained in a ball of radius $2d$ around $x$. Let $c:= 2(1-3m/(4\pi(2d)^3))^{1/3}$, so that the spherical shell between radii $cd$ and $2d$ has volume $m$. Then, again by the `bathtub principle', if $x\in B$,
$$
\int_{\R^3} |x-y|^{\alpha-2} \rho(y)\,dy \geq \int_{B_{2d}(x) \setminus B_{cd}(x)} |x-y|^{\alpha-2}\,dy = C' d^{\alpha+1} \left(2^{\alpha+1} - c^{\alpha+1}\right) \quad\text{if}\ \alpha<2
$$
and
$$
\int_{\R^3} |x-y|^{\alpha-2} \rho(y)\,dy \leq \int_{B_{2d}(x) \setminus B_{cd}(x)} |x-y|^{\alpha-2}\,dy = C' d^{\alpha+1} \left( 2^{\alpha+1} - c^{\alpha+1} \right) \quad\text{if}\ \alpha>2\,.
$$
To complete the proof of the corollary, we note that
$$
d^{\alpha+1} \left( 2^{\alpha+1} - c^{\alpha+1} \right) \approx m \qquad\text{if}\ m\leq 1
$$
and
$$
d^{\alpha+1} \left( 2^{\alpha+1} - c^{\alpha+1} \right) \approx m^{(\alpha+1)/3}
\qquad\text{if}\ m\geq 1
$$
where $\approx$ means that there are upper and lower bounds on the ratio with finite, positive constants depending only on $\alpha$.
\end{proof}

%%%%%%%%%%%%%%%%%%%%%%%%%%%%%%%%%%%%%%%%%%%%%%%%%%%%%%%%

\section{Proof of the main results, Theorems \ref{main1} and \ref{main}}

\subsection{Outline of the proof}

Both proofs of Theorems \ref{main1} and \ref{main} rely on Proposition~\ref{sobconst}, which implies that $\Delta \phi=0$ on $\{\phi=\mu\}$. In order to apply this proposition, we need to verify that $\phi\in W^{2,1}_\loc(\R^3)$ and in the following lemma we will show that in fact $\phi\in W^{2,p}_\loc(\R^3)$ for any $p<\infty$.

We recall that given $\rho$, $\phi$ is defined by \eqref{eq:pot}.

\begin{lemma}\label{regularity}
For any $\rho\in L^\infty$ with $|x|^\alpha\rho\in L^1(\R^3)$ one has $\phi\in W^{2,p}_\loc(\R^3)$.
\end{lemma}

\begin{proof}
It is easy to see that $\phi$ is a continuous function, so it is enough to show that $\partial_i\partial_j\phi\in L^p_\loc(\R^3)$ for any $i,j$. We decompose $\phi = \phi_{-1}+\phi_\alpha$ as in the proof of Lemma \ref{upperbound}. We have $-\Delta\phi_{-1} = 4\pi\rho$ in the sense of distributions. Since $\rho\in L^p(\R^3)$ for any $p<\infty$, we deduce from the Calderon--Zygmund inequality (see, e.g., \cite[Thm. 9.9]{GT}) that $\partial_i\partial_j\phi_{-1}\in L^p(\R^3)$.

In fact, let us give an elementary proof of the weaker fact that $\partial_i\partial_j\phi_{-1}\in L^2(\R^3)$ (which, however, is sufficient for our application of Proposition \ref{sobconst}).  Since $\rho\in L^2(\R^3)$ and since the Fourier transform of $e^{-|x|}/|x|$ is a constant times $1/(1+p^2)$, we deduce that $e^{-|x|}/|x| * \rho\in H^2(\R^3)$. Moreover, since $|\partial_i\partial_j (1-e^{-|x|})/|x|| \lesssim |x|^{-2}(1+|x|)^{-1}$ and $\rho\in L^p(\R^3)$ for $3<p<\infty$, we conclude that $\partial_i\partial_j (1-e^{-|x|})/|x|*\rho\in L^\infty$.

For the $\phi_\alpha$ piece we use the fact that $|\partial_i\partial_j\phi|\leq C |x|^{\alpha-2}$. Using $\rho\in L^\infty$ and $|x|^\alpha\rho\in L^1$ we again deduce that $\partial_i\partial_j\phi_\alpha\in L^\infty$.
\end{proof}

%%%%%%%%%%%%%%%%%%%%%%%%%%%%%%%%%%%%%%%%%%%%%%%%%%%%%%%%

\subsection{The regime of small mass -- Proof of Theorem \ref{main1}}

Let us consider the minimization problem
$$
E^*_\alpha := \inf\left\{ \mathcal E[\rho]:\ \rho\geq 0\,,\ \int_{\R^3}\rho(x)\,dx = 1 \right\} \,.
$$
We know from \cite{CCP,SST} that a minimizer exists and from \cite{CDM} that any minimizer is bounded, so
$$
M^*(\alpha) := \sup\left\{ \|\rho\|_\infty^{-1} :\ \rho \ \text{minimizer for}\ E^*_\alpha \right\} >0\,.
$$
The following simple fact is essentially contained in \cite[Proof of Thm. 1.1(i)]{BCT}, but for the sake of completeness we provide a proof.

\begin{lemma}\label{energysmall}
If $m\leq M^*(\alpha)$, then $E_\alpha(m) = m^2 E_\alpha^*$.
\end{lemma}

\begin{proof}
We begin by proving $\geq$ for \emph{any} $m>0$. For any $\rho$ with $0\leq\rho\leq1$ and $\int_{\R^3}\rho\,dx =m$ we can take $\rho/m$ as a trial state for the $E_\alpha^*$ problem and obtain, by homogeneity,
$$
E_\alpha^* \leq \mathcal E_\alpha[\rho/m] = \mathcal E_\alpha[\rho]/m^2 \,.
$$
Taking the infimum over all such $\rho$ we obtain $E_\alpha^* \leq E_\alpha(m)/m^2$, which is $\geq$ in the lemma.

For the converse inequality, let $m<M^*(\alpha)$ and choose a minimizer $\rho_*$ for $E^*_\alpha$ with $\|\rho_*\|_\infty^{-1} \geq m$. Then $m\rho_*$ is an admissible trial state for the $E_\alpha(m)$ problem and we obtain, again by homogeneity,
$$
E_\alpha(m) \leq \mathcal E_\alpha[m\rho^*] = m^2 \mathcal E_\alpha[\rho_*] = m^2 E_\alpha^* \,,
$$
which is $\leq$ in the lemma. The equality extends to $m=M^*(\alpha)$ by continuity.
\end{proof}

\begin{proof}[Proof of Theorem \ref{main1}]
Let $m\leq M^*(\alpha)$ and let $\rho$ be a minimizer for $E_\alpha(m)$. We infer from Lemma \ref{energysmall} that $E_\alpha'(m)=2m E_\alpha^*$ (for $m=M^*(\alpha)$ this holds for the left derivative) and therefore, by Lemma \ref{chempot}, that $\mu = 2mE_\alpha^*$. Thus,
\begin{align*}
\mu m & = 2m^2 E_\alpha^* = 2 E_\alpha(m) = \int_{\R^3} \phi(x)\rho(x)\,dx \\
& = \mu \int_{\R^3}\rho(x)\,dx + \int_{\R^3} (\phi(x)-\mu) \rho(x)\,dx \\
& = \mu m + \int_{\R^3} (\phi(x)-\mu) \rho(x)\,dx \,,
\end{align*}
that is,
$$
\int_{\R^3} (\phi(x)-\mu) \rho(x)\,dx = 0 \,.
$$
According to the Euler equation \eqref{eq:eeq}, $\phi\leq\mu$ a.e. on $\{\rho>0\}$. Therefore we conclude that
$$
\phi = \mu \qquad \text{a.e. on}\ \{\rho >0\} \,.
$$
By Proposition \ref{sobconst} and Lemma \ref{regularity} we deduce that
\begin{equation}
\label{eq:shappl1}
-\Delta\phi = 0 \qquad \text{a.e. on}\ \{\rho >0\} \,.
\end{equation}

On the other hand, by an explicit calculation, we have
\begin{equation}
\label{eq:eeqlap}
-\Delta\phi(x) = 4\pi \rho(x) -\alpha(\alpha+1) \int_{\R^3} |x-y|^{\alpha-2}\rho(y)\,dy \,.
\end{equation}
Using Corollary \ref{diamcor} we can bound for almost every $x\in\R^3$ with $\rho(x)=1$
\begin{equation}
\label{eq:diamappl1}
4\pi \rho(x) -\alpha(\alpha+1) \int_{\R^3} |x-y|^{\alpha-2}\rho(y)\,dy \geq 4\pi - C \max\{m,m^{(\alpha+1)/3}\} \,.
\end{equation}
For $\alpha\leq 2$ the maximum can be replaced by $m^{(\alpha+1)/3}$. In any case, if $|\{\rho=1\}|>0$, we learn from \eqref{eq:shappl1}, \eqref{eq:eeqlap} and \eqref{eq:diamappl1} that $4\pi \leq C \max\{m,m^{(\alpha+1)/3}\}$, that is, $m\geq \max\{4\pi/C, (4\pi/C)^{3/(\alpha+1)}\}$ (and $m\geq (4\pi/C)^{3/(\alpha+1)}$ if $\alpha\leq 2$). This proves the theorem.
\end{proof}

%%%%%%%%%%%%%%%%%%%%%%%%%%%%%%%%%%%%%%%%%%%%%%%%%%%

\subsection{The regime of large mass -- Proof of Theorem \ref{main}}

According to the Euler equation \eqref{eq:eeq}, $\phi=\mu$ a.e. on $\{0<\rho<1\}$ and so again by Proposition \ref{sobconst} and Lemma~\ref{regularity} we deduce that
\begin{equation}
\label{eq:shappl2}
-\Delta\phi = 0 \qquad \text{a.e. on}\ \{0<\rho <1\} \,.
\end{equation}

As in the previous proof we will compare this with the formula for $-\Delta\phi$ from \eqref{eq:eeqlap}. Using Corollary \ref{diamcor} we can bound for almost every $x\in\R^3$ with $0<\rho(x)<1$
\begin{equation}
\label{eq:diamappl2}
4\pi \rho(x) -\alpha(\alpha+1) \int_{\R^3} |x-y|^{\alpha-2}\rho(y)\,dy \leq 4\pi - C \min\{m,m^{(\alpha+1)/3}\} \,.
\end{equation}
For $\alpha\geq 2$ the minimum can be replaced by $m^{(\alpha+1)/3}$. In any case, if $|\{0<\rho<1\}|>0$, we deduce from \eqref{eq:shappl2}, \eqref{eq:eeqlap} and \eqref{eq:diamappl2} that $4\pi \geq C \min\{m,m^{(\alpha+1)/3}\}$, that is, $m\leq \max\{4\pi/C, (4\pi/C)^{3/(\alpha+1)}\}$ (and $m\geq (4\pi/C)^{3/(\alpha+1)}$ if $\alpha\geq 2$). This proves the theorem.\qed

%%%%%%%%%%%%%%%%%%%%%%%

\subsection{Extension to more general kernels}\label{sec:extensions}

In this subsection we consider the minimization problem \eqref{eq:energygen} with the energy function \eqref{eq:energyfcnlgen} involving a general kernel $k$ satisfying \eqref{eq:kernelgen}.

\begin{proposition}
Assume that the distributional Laplacian $\Delta k$ is bounded from below by a locally integrable function and satisfies $\liminf_{|x|\to\infty} \Delta k >0$. Then there is an $m_*<\infty$ such that for $m>m_*$ any minimizer $\rho$ for $E(m)$ satisfies $|\{\rho=1\}|=m$.
\end{proposition}

\begin{proof}
We argue as in the proof of Theorem \ref{main} and therefore it suffices to show that $\Delta\phi>0$ if $m$ is large enough. By assumption there is a $\kappa\in L^1_{\rm loc}(\R^3)$ and constants $c>0$ and $R>0$ such that $\Delta k\geq \kappa$ and $\kappa(x)\geq c$ if $|x|\geq R$. Thus,
\begin{align*}
\Delta\phi(x) & \geq \int_{\R^3} \kappa(x-y)\rho(y)\,dy \geq - \int_{\{|x-y|<R\}} \kappa_-(x-y)\rho(y)\,dy + c \int_{\{|x-y|\geq R\}} \rho(y)\,dy \\
& \geq - \int_{\{|z|<R\}} \kappa_-(z)\,dz + c \left( m - (4\pi/3) R^3 \right).
\end{align*} 
The right side if clearly positive if $m$ is large enough.
\end{proof}

This proposition applies, for instance, to $k(x) = |x|^{-\beta}+|x|^\alpha$ with $0<\beta<1$ and $\alpha\geq 2$. The same result holds also for $0<\alpha<2$, with a proof that is parallel to that of Theorem \ref{main} and, as in that case, one needs in addition a bound on the diameter of the support of $\rho$. Such bound can be obtained by the same way as in Section \ref{sec:diameter} and we omit the details.

%%%%%%%%%%%%%%%%%%%%%%%%%%%%%%%%%%%%%%%%%%%%%%%%%%%%%%%%%%%%%%%%%%%%%%%%%%%%%%%%

\appendix

\section{A theorem about subharmonic functions. I}

In this appendix we complement Corollary \ref{sh} by a similar result under different assumptions, which is interesting in its own right. While Corollary \ref{sh} concerns Sobolev functions and has a relatively straightforward proof, we now discuss the case of continuous functions without any integrability assumptions on derivatives. The proof is technically significantly more difficult and we are greatly indebted to Luis Silvestre for showing us how to extend the result from $C^{1,1}$ functions to continuous functions. A further generalization will be discussed in the following appendix.

We recall that, if $u$ is a continuous, real function on an open set $\Omega$, we say that $-\Delta u \leq f$ in $\Omega$ in \emph{viscosity sense} if for any $x\in\Omega$ and any $\phi\in C^2(\Omega)$ for which $u-\phi$ has a local maximum at $x$ one has $-\Delta\phi(x)\leq f(x)$.

\begin{proposition}\label{shcont}
Let $\Omega\subset\R^d$, $d\geq 1$, be an open set and assume that $u\in C(\Omega)$ satisfies $-\Delta u\leq -\epsilon$ in $\Omega$ in viscosity sense for some $\epsilon>0$. Then $|\{x\in\Omega:\ u(x)=\tau\}|=0$ for any $\tau\in\R$.
\end{proposition}

This proposition has an elementary proof under the additional assumption $u\in C^{1,1}(\Omega)$, but remarkably it also holds without this assumption. We also note that the statement is wrong if we only assume $-\Delta u\leq 0$ in $\Omega$, as the example $u(x) =\max\{x_1,0\}$ shows.

For the proof we define for any set $\Omega\subset\R^d$ and any $u\in C(\Omega)$
\begin{align*}
\overline{\Theta}(u,\Omega)(x):= & \inf\left\{ A\geq 0:\ \text{there is a}\ p\in\R^d\ \text{such that for all}\ y\in\Omega, \right. \\
&  \qquad \qquad\qquad \left. u(y) \leq u(x) +p\cdot(y-x) + (A/2)|y-x|^2 \right\}
\end{align*}
with the convention that $\inf\emptyset=+\infty$. We shall use the following deep result \cite[Ch. 7]{CC} (see also \cite[Prop. 3.1]{ASS}).

\begin{lemma}\label{cc}
Let $B\subset\R^d$ be a ball and $B'$ a concentric ball with twice the radius. If $u\in C(B')$ satisfies $-\Delta u \leq 0$ in $B'$ in viscosity sense, then
$$
\lim_{t\to\infty} \left|\{ x\in B:\ \overline{\Theta}(u,B')(x)>t\}\right| = 0 \,.
$$ 
\end{lemma}

In fact, there are bounds on how fast the measure of the set in the lemma tends to zero, but they are not important for us.

\begin{proof}
Replacing $u$ by $u-\tau$ we may assume that $\tau=0$. We will show that for every ball $B$ such that $\overline{B'}\subset\Omega$ (where $B'$ denotes the concentric ball with twice the radius) one has $|\{u =0\}\cap B|=0$. This will clearly imply the result.

We argue by contradiction and assume that $|\{u =0\}\cap B|>0$. According to Lemma~\ref{cc} we can choose $t$ so large that $|\{\overline{\Theta}(u,B')>t\}\cap B|<|\{u=0\}\cap B|$. This implies that $\{\overline{\Theta}(u,B')\leq t\}\cap \{u=0\}\cap B$ has positive measure and we choose $x$ to be a Lebesgue point of this set and assume, after a translation, that $x=0$. Thus, we have $u(0) =0$ and there is a $p\in\R^d$ such that
$$
u(y) \leq p\cdot y + (t/2)|y|^2
\qquad\text{for all}\ y\in B' \,.
$$
Let $B_r=\{y: |y|<r\}$. Since $0$ is a Lebesgue point of $\{u=0\}\cap B$, we have $|B_r|^{-1} |\{u=0\}\cap B_r|\to 1$ as $r\to 0$ and therefore $p=0$. (Indeed, otherwise $u(y)$ would be negative in the cone $\{ y\in\R^d:\ p\cdot y\leq -\delta|y|\,,\ |y| <2\delta/t\}$ where $\delta<|p|$ is a fixed constant.) Therefore, we can bound for every $r>0$ such that $B_r\subset B'$
\begin{align}
\label{eq:parabolabd}
\int_{B_r} u(y)\,dy & \leq \int_{\{u\neq 0\}\cap B_r\}} u(y)\,dy  \leq (t/2) r^2 \left| \{u\neq 0\}\cap B_r \right|.
\end{align}

On the other hand, let us derive a lower bound on the left side. Using the Green's function for the ball we find that for any $C^2$ function $v$ on $B_\rho$ one has
$$
v(0) = - c_d \int_{B_\rho} \left(\frac{1}{|y|^{d-2}} - \frac{1}{\rho^{d-2}} \right)\Delta v(y)\,dy
+ \frac{(d-2)c_d}{\rho^{d-1}} \int_{|y|=\rho} v(y)\,d\sigma(y) \,.
$$
with $c_d = ((d-2)|\Sph^{d-1}|)^{-1}$. (The formula for $d=1,2$ is similar and is omitted.) Multiplying by $\rho^{d-1}$ and integrating with respect to $\rho$, we obtain
$$
|B_r| v(0) = -\frac1{d-2} \int_{B_r} \left( \frac1d\frac{r^d-|y|^d}{|y|^{d-2}} - \frac12 \left(r^2-|y|^2\right) \right)\Delta v(y)\,dy + \int_{B_r} v(y)\,dy \,. 
$$
We apply this inequality to $v=\eta_\delta* u$, where $\eta_\delta(y) = \delta^{-d}\eta(y/\delta)$ with some mollifier $\eta\geq 0$ and note that $-\Delta v\leq -\epsilon$. Using this inequality for the first term on the right side and then letting $\delta\to 0$, we obtain, since $u(0)=0$,
$$
0 \leq -\frac{\epsilon}{d-2} \int_{B_r} \left( \frac1d\frac{r^d-|y|^d}{|y|^{d-2}} - \frac12 \left(r^2-|y|^2\right) \right) dy + \int_{B_r} u(y)\,dy \,,
$$
that is,
\begin{equation}
\label{eq:sshbound}
\int_{B_r} u(y)\,dy \geq \epsilon c_d' r^2 |B_r|
\end{equation}
with a constant $c_d'>0$ depending only on $d$.

Comparing \eqref{eq:parabolabd} and \eqref{eq:sshbound}, we find that for every $r>0$ such that $B_r\subset B'$,
$$
\frac{|\{u\neq 0\}\cap B_r|}{|B_r|} \geq \frac{2c_d'\epsilon}{t}
$$
This contradicts the fact that $0$ is a Lebesgue point of the set $\{u=0\}\cap B$. This proves the theorem.
\end{proof}

%%%%%%%%%%%%%%%%%%%%%%%%%%%%%%%%%%%%%%%%%%%%%%%%%%%%%%%%%%%%%%%%%%%%%%%%%%%%%%%%

\section{A theorem about subharmonic functions. II}

In this appendix we generalize Corollary \ref{sh} to general subharmonic functions. We learned the argument from Mikhail Sodin, to whom we are grateful.

\subsection{Statement of the result and outline of the proof}

If $u$ is a subharmonic function on an open set $\Omega\subset\R^d$, then the distribution $\Delta u$ is non-negative and therefore extends to a non-negative, locally finite, regular Borel measure on $\Omega$, which we denote by $\mu_u$. (Here, `regular' means that $\mu_u(A)=\inf\{\mu(O):\ O\supset A \ \text{open}\}$ and $\mu_u(A) = \sup\{\mu(K):\ K\subset A\ \text{compact} \}$ for every measurable set $A\subset\Omega$. Moreover, for us `measurable' always means Borel measurable.)

\begin{theorem}\label{mainsh}
Let $\Omega\subset\R^d$ be open and $u\in L^1_{\loc}(\Omega)$ subharmonic such that $\mu_u(E)>0$ for any measurable $E\subset\Omega$ with $|E|>0$. Then $|\{u=\tau\}|=0$ for any $\tau\in\R$.
\end{theorem}

The proof of this theorem hinges on the following two results. The first one extends an argument of Eremenko--Sodin \cite{ErSo} (see also their references to earlier work by \O ksendal) to arbitrary dimensions.

\begin{proposition}\label{inf}
Let $\Omega\subset\R^d$ be open and $u\in L^1_{\loc}(\Omega)$ be subharmonic and non-negative. Then
$$
\mu_u \left( \left\{ x\in\Omega:\ \lim_{r\to 0} \frac{|\{ y\in B_r(x):\ u(y)=0\}|}{|B_r|} = 1 \right\} \right) = 0 \,.
$$
\end{proposition}

In other words, $\Delta u$ vanishes on the set of Lebesgue points of $\{u=0\}$. In fact, our proof shows that there is an $\epsilon_d>0$, depending only on $d$, such that
$$
\mu_u \left( \left\{ x\in\Omega:\ \liminf_{r\to 0} \frac{|\{ y\in B_r(x):\ u(y)=0\}|}{|B_r|} > 1-\epsilon_d \right\} \right) = 0 \,.
$$

The second ingredient in our proof of Theorem \ref{mainsh} is a special case of Grishin's lemma, see \cite{Gr} and also \cite{So} and references therein.

\begin{proposition}\label{grishin}
Let $\Omega\subset\R^d$ be open and let $u,v\in L^1_{\loc}(\Omega)$ be subharmonic with $v\geq u$. Then
$$
\left(\mu_v-\mu_u\right)|_{\{u=v>-\infty\}} \geq 0 \,. 
$$
\end{proposition}

Assuming these two propositions we now show how they imply Theorem \ref{mainsh}.

\begin{proof}
[Proof of Theorem \ref{mainsh}]
Replacing $u$ by $u-\tau$, we may assume that $\tau=0$. The function $v:=u_+$ is subharmonic and, setting
$$
E=\left\{ x\in\Omega:\ \lim_{r\to 0} \frac{|\{ y\in B_r(x):\ u(y)\leq 0\}|}{|B_r|} = 1 \right\},
$$
we deduce from Proposition \ref{inf} that $\mu_v(E)=0$. Therefore, if
$$
E'= \left\{ x\in\Omega:\ \lim_{r\to 0} \frac{|\{ y\in B_r(x):\ u(y)= 0\}|}{|B_r|} = 1 \right\} \,,
$$
then $E'\subset E$ and therefore also $\mu_v(E')=0$. On the other hand, since $v\geq u$, Proposition \ref{grishin} implies that 
$$
\left( \mu_v - \mu_u \right)|_{\{u\geq 0\}} \geq 0 \,.
$$
Therefore, if
$$
E''=\{ x\in\Omega:\ u(x)=0\} \,,
$$
then $\mu_u(E'\cap E'') \leq \mu_v(E'\cap E'')\leq \mu_v(E') = 0$, so $\mu_u(E'\cap E'')=0$. The assumed strict subharmonicity of $u$ therefore implies that $|E'\cap E''|=0$. But, by Lebesgues's theorem, $|E''\setminus E'|=0$ and therefore $|E''|=0$, as claimed.
\end{proof}

\subsection{Tools for the proof of the propositions}

Thus, it remains to prove Propositions \ref{inf} and \ref{grishin}. Their proofs rely on two ingredients. The first one is a variant of Green's formula and the second one a general result about measures. For the first result we recall that the Green's function for the unit ball $B\subset\R^d$ is given by
$$
G(x,y) = 
\begin{cases}
\frac{1}{(d-2)|\Sph^{d-1}|} \left( \frac{1}{|x-y|^{d-2}} - \frac{1}{(1-2x\cdot y + x^2 y^2)^{(d-2)/2}} \right) & \text{if}\ d\geq 3 \,,\\
\frac{1}{2\pi} \left( \ln \frac{1}{|x-y|} - \ln \frac{1}{(1-2x\cdot y + x^2 y^2)^{1/2}} \right) & \text{if}\ d= 2 \,.
\end{cases}
$$
Therefore, if $u$ is, say, $C^2$ in the unit ball and continuous up to the boundary, then we have Green's representation formula
$$
u(x) = - \int_B G(x,y)\Delta u(y)\,dy - \int_{\Sph^{d-1}} \frac{\partial G}{\partial\nu_y}(x,y) u(y)\,d\sigma(y) \,.
$$
One consequence of this formula is that, if $u$ is subharmonic,
\begin{equation}
\label{eq:green1}
u(x) \leq - \int_{\partial B} \frac{\partial G}{\partial\nu_y}(x,y) u(y)\,d\sigma(y) = \frac{1}{|\Sph^{d-1}|} \int_{\Sph^{d-1}} \frac{1-x^2}{(1-2x\cdot y + x^2)^{d/2}} u(y)\,d\sigma(y) \,.
\end{equation}
Using a simple density argument, based for instance on \cite[Thm. 9.3]{LL}, this inequality extends to any (not necessarily smooth) subharmonic function in the unit ball.

Next, we assume again that $u$ is, say, $C^2$ in $B$ and continuous on $\overline B$ and take $x=0$ in Green's representation formula. For $d\geq 3$ we obtain
$$
u(0) = - \frac{1}{(d-2)|\Sph^{d-1}|} \int_B \left( \frac{1}{|y|^{d-2}} - 1 \right) \Delta u(y)\,dy +
\frac{1}{|\Sph^{d-1}|} \int_{\Sph^{d-1}} u(y)\,d\sigma(y) \,.
$$
We note that for any $\alpha>0$,
$$
\alpha^{-1} \left( |y|^{-\alpha} - 1\right) = \alpha^{-1} \int_0^\infty \chi_{\{t<|y|^{-\alpha}-1\}} \,dt = \int_0^1 \chi_{\{|y|<\rho\}} \frac{d\rho}{\rho^{\alpha+1}} \,.
$$
Thus,
\begin{equation}
\label{eq:green2}
u(0) = - \frac{1}{|\Sph^{d-1}|} \int_0^1 \mu_u(B_\rho) \frac{d\rho}{\rho^{d-1}} + \frac{1}{|\Sph^{d-1}|} \int_{\Sph^{d-1}} u(y)\,d\sigma(y) \,.
\end{equation}
This formula is also true if $d=2$. Moreover, by a density argument, based for instance on \cite[Thm. 9.3]{LL}, it extends to any function $u$ which is subharmonic in a neighborhood of the unit ball.

The second ingredient in the proof of the propositions is the following measure theoretic result.

\begin{lemma}\label{posmeas}
Let $\mu$ be a signed, real, locally finite, regular Borel measure on $\R^d$ and let $E\subset\R^d$ be a measurable set such that for any $x\in E$,
$$
\limsup_{r\to 0} \frac{\mu(B_r(x))}{|B_r|}\geq 0 \,.
$$
Then $\mu|_E\geq 0$.
\end{lemma}

\begin{proof}[Proof of Lemma \ref{posmeas}]
Let $A\subset E$ be measurable. We show that $\mu(A)\geq 0$. We may assume that $A$ is bounded. (Otherwise we choose a tiling of $\R^d$ by half-open, disjoint cubes $Q_j$ and consider $A\cap Q_j$ for each fixed $j$.)

Let $\epsilon>0$. By assumption, for any $x\in A$ there is a monotone decreasing sequence $(r_n(x))_{n\in\N} \subset(0,1]$, tending to zero, such that $\mu(B_{r_n(x)}(x))\geq -\epsilon|B_{r_n(x)}|$ for all $n$. We apply \cite[Thm. 2.8]{Ma}, which is a consequence of the Besicovich covering theorem, to the family $\{ \overline{B_{r_n(x)}(x)}:\ x\in A\,, n\in\N\}$ and obtain a countable subfamily of disjoint balls $B_j$ such that $|\mu|(A\setminus\bigcup_j B_j)=0$. We write
$$
\mu(A) = \mu(\bigcup B_j) + \mu(A\setminus \bigcup_j B_j)
$$
and bound
$$
\mu(\bigcup B_j) = \sum_j \mu(B_j) \geq -\epsilon \sum_j |B_j| = - \epsilon |\bigcup_j B_j| \geq -\epsilon C\,,
$$
where $C= |\{ x\in\R^d:\ \dist(x,A)\leq 1\}|$, which is finite since $A$ is bounded. On the other hand,
$$
\left| \mu(A\setminus \bigcup_j B_j) \right| \leq |\mu|(A\setminus \bigcup_j B_j) = 0 \,,
$$
so $\mu(A)\geq -\epsilon C$. Since $\epsilon>0$ is arbitrary, we conclude that $\mu(A)\geq 0$, as claimed.
\end{proof}

\subsection{Proof of the propositions}

With these tools at hand we can now prove Propositions \ref{inf} and \ref{grishin} and thereby complete the proof of Theorem \ref{mainsh}.

\begin{proof}[Proof of Proposition \ref{inf}]
Let $\epsilon_d=(1-2^{-d})/(2^{d+2}\, 3)$ and
$$
E := \left\{ x\in\Omega:\ \liminf_{r\to 0} \frac{|\{ y\in B_r(x):\ u(y)=0\}|}{|B_r|} > 1-\epsilon_d \right\}.
$$
We shall show that for any $x\in E$,
$$
\liminf_{r\to0} \frac{\mu_u(B_r)}{|B_r|} = 0 \,.
$$
This fact, together with Lemma \ref{posmeas} (applied to $\mu=-\mu_u$), yields that $\mu_u|_E\leq 0$. On the other hand, by subharmonicity $\mu_u\geq 0$, which implies that $\mu_u(E)=0$. (In fact, since one knows $\mu_u\geq 0$, the proof of Lemma \ref{posmeas} can be somewhat abbreviated.)

Fix $x\in E$ and $\epsilon\in (1- \liminf_{r\to 0} |\{ y\in B_r(x):\ u(y)=0\}|/|B_r|,\epsilon_d)$. We choose $r_0>0$ such that $\dist(x,\Omega^c)< r_0$, $M_0:= \sup_{|y-x|=r_0} u(x)<\infty$ and
$$
|\{ y\in B_r(x) : u(y)\neq 0 \}| \leq \epsilon |B_r|
\qquad\text{for all}\ 0<r\leq r_0 \,. 
$$
Moreover, let
$$
\theta(r)= \mathcal H^{d-1}(\{\omega\in\Sph^{d-1}:\ u(x+r\omega) \neq 0\})/|\Sph^{d-1}| \,.
$$
We claim that for any $0<r\leq r_0$ there is an $r'\in(r/2,r)$ such that $\theta(r') \leq \epsilon/(1-2^{-d})=:\eta$. In fact, if we had $\theta(s)>\eta$ for all $s\in (r/2,r)$, we had
\begin{align*}
\epsilon |B_r| & = \eta |\Sph^{d-1}| \int_{r/2}^r s^{d-1}\,ds < |\Sph^{d-1}| \int_{r/2}^r \theta(s) s^{d-1}\,ds \leq |\Sph^{d-1}| \int_0^r \theta(s) s^{d-1}\,ds \\
& = |\{ y\in B_r(x):\ u(y)\neq 0\}| \leq \epsilon |B_r| \,,
\end{align*}
a contradiction.

Applying this claim iteratively we obtain a sequence $(r_k)$ with $1/4\leq r_{k+1}/r_k\leq 1/2$ and $\theta(r_k)\leq\eta$. Let $M_k:= \sup_{|y-x|=r_k} u(x)$. Green's representation formula \eqref{eq:green2} together with the fact that $u(x)\geq 0$ implies that
$$
\int_0^{r_k} \mu_u(B_\rho(x)) \frac{d\rho}{\rho^{d-1}} \leq \int_{\Sph^{d-1}} u(x+r_k\omega)\,d\sigma(\omega) \,.
$$
Bounding the right side from above and the left side from below, we obtain
$$
\frac{2^{d-2}-1}{d-2} \mu_u(B_{r_k/2}(x)) r_k^{-d+2} = \mu_u(B_{r_k/2}(x)) \int_{r_k/2}^{r_k} \frac{d\rho}{\rho^{d-1}} \leq |\Sph^{d-1}| M_k \,.
$$
(If $d=2$, the constant $(2^{d-2}-1)/(d-2)$ is replaced by $\ln 2$.) This is the same as
$$
\frac{\mu_u(B_{r_k/2}(x))}{|B_{r_k/2}|} \leq \frac{d-2}{2^{d-2}-1}
 d \, 2^d \frac{M_k}{r_k^2} \,.
$$
Thus,
$$
\liminf_{r\to0} \frac{\mu_r(B_r)}{|B_r|} \leq \frac{d-2}{2^{d-2}-1}
 d \, 2^d \liminf_{k\to\infty} \frac{M_k}{r_k^2} \,,
$$
and the proposition will follow if we can show that $M_k/r_k^2\to 0$ as $k\to\infty$.

In order to show this, we observe that, after rescaling, inequality \eqref{eq:green1} implies that for any $y$ with $|y-x|=r_{k+1}$,
$$
u(y) \leq \frac{1}{|\Sph^{d-1}|} \int_{\Sph^{d-1}} \frac{r_k^{d-2} \left(r_k^2-r_{k+1}^2\right)}{(r_k^2 - 2r_k (y-x)\cdot\omega + r_{k+1}^2)^{d/2}} u(x+r_k\omega) \,d\sigma(\omega) \,.
$$ 
Thus, recalling that $u\geq 0$,
$$
u(y) \leq \frac{r_k^{d-2}(r_k^2-r_{k+1}^2)}{(r_k-r_{k+1})^d} \frac{1}{|\Sph^{d-1}|} \int_{\Sph^{d-1}} u(x+r_k\omega) \,d\sigma(\omega) \,.
$$
We bound
$$
\frac{1}{|\Sph^{d-1}|} \int_{\Sph^{d-1}} u(x+r_k\omega) \,d\omega \leq M_k \theta(r_k) \leq M_k \eta \,.
$$
Moreover, since $r_{k+1}/r_k\leq 1/2$ we have
$$
\frac{r_k^{d-2}(r_k^2-r_{k+1}^2)}{(r_k-r_{k+1})^d} = \frac{r_k^{d-2}(r_k+r_{k+1})}{(r_k-r_{k+1})^{d-1}} \leq \frac{3/2}{(1/2)^{d-1}} = 2^{d-2} \,3 \,. 
$$
Since this bound is valid for any $y$ with $|y-x|=r_{k+1}$ we conclude that
$$
M_{k+1} \leq 2^{d-2}\, 3 \eta M_k \,.
$$
Since $r_{k+1}/r_k \geq 1/4$, this implies that
$$
\frac{M_{k+1}}{r_{k+1}^2} \leq 2^{d-2}\, 3 \eta \frac{r_k^2}{r_{k+1}^2} \frac{M_k}{r_k^2} \leq 2^{d+2}\, 3\eta \frac{M_k}{r_k^2} \,.
$$
Iterating this, we obtain
$$
\frac{M_{k+1}}{r_{k+1}^2} \leq \left( 2^{d+2}\, 3\eta \right)^{k+1} \frac{M_0}{r_0^2} \,.
$$
Since
$$
2^{d+2}\,3\eta = \frac{2^{d+2}\,3\epsilon}{1-2^{-d}} < \frac{2^{d+2}\,3\epsilon_d}{1-2^{-d}} = 1\,,
$$
we infer that $M_{k+1}/r_{k+1}^2 \to 0$ as $k\to\infty$, which concludes the proof.
\end{proof}

\begin{proof}[Proof of Proposition \ref{grishin}]
Let $x\in\Omega$ such that $u(x)=v(x)>-\infty$. We shall show that there is a sequence $(r_n)\subset(0,\infty)$, tending to zero, such that $(\mu_v-\mu_u)(B_{r_n}(x))\geq 0$. According to Lemma \ref{posmeas} this will imply that $\mu_v-\mu_u\geq 0$ on $\{u=v>-\infty\}$.

For any $R<\dist(x,\Omega^c)$, we have according to \eqref{eq:green2} and rescaling,
$$
\int_0^R \left( \mu_v(B_\rho(x))-\mu_u(B_\rho(x))\right) \frac{d\rho}{\rho^{d-1}} = \int_{\Sph^{d-1}} \left(v(x+R\omega)-u(x+R\omega)\right)d\sigma(\omega) \,.
$$
Since the right side is non-negative by assumption, there is an $0<r_0<\dist(x,\Omega^c)$ such that $\mu_v(B_{r_0}(x))-\mu_u(B_{r_0}(x))\geq 0$. We now use the formula with $R$ replaced by $r_0/2$. Again, the right side is non-negative and therefore there is an $0<r_1<r_0/2$ such that $\mu_v(B_{r_1}(x))-\mu_u(B_{r_1}(x))\geq 0$. Continuing in this way we obtain a sequence with the claimed properties. This concludes the proof of the proposition.
\end{proof}

%%%%%%%%%%%%%%%%%%%%%%%%%%%%%%%%%%%%%%%%%%%%%%%%%%%%%%%%%%%%%%%%%%%%%%%%%%%%%%%%

\bibliographystyle{amsalpha}

\end{document}